\documentclass[12pt]{article}
\usepackage{amsmath, amsthm, amsfonts}
\usepackage{amssymb}
\usepackage{amscd}
\usepackage{verbatim}
\usepackage{mathrsfs}
\usepackage{esint}  
\usepackage{graphicx,color}
\usepackage{texdraw}
\usepackage{hyperref}
\begin{document}
\pretolerance=10000
\def\Ln{\mathcal{L}^{n}}
\def\L1{\mathcal{L}^{1}}
\def\Hn{\mathcal{H}^{n-1}}
\def\HN{\mathcal{H}^{N-1}}
\def\H1{\mathcal{H}^{1}}
\def\O{\Omega}
\newcommand{\ud}{{\mathrm d}}
\newcommand{\half}{\frac12}
\newcommand{\eqdef}{\stackrel{{\mathrm def}}{=}}
\newcommand{\M}{{\mathcal M}}
\newcommand{\loc}{{\mathrm{loc}}}
\newcommand{\dx}{\,\mathrm{d}x}
\newcommand{\core}{C_0^{\infty}(\Omega)}
\newcommand{\sob}{W^{1,p}(\Omega)}
\newcommand{\sobloc}{W^{1,p}_{\mathrm{loc}}(\Omega)}
\newcommand{\merhav}{{\mathcal D}^{1,p}}
\newcommand{\be}{\begin{equation}}
\newcommand{\ee}{\end{equation}}
\newcommand{\mysection}[1]{\section{#1}\setcounter{equation}{0}}
\newcommand{\bea}{\begin{eqnarray}}
\newcommand{\eea}{\end{eqnarray}}
\newcommand{\bean}{\begin{eqnarray*}}
\newcommand{\eean}{\end{eqnarray*}}
\newcommand{\thkl}{\rule[-.5mm]{.3mm}{3mm}}
\newcommand{\cw}{\stackrel{D}{\rightharpoonup}}
\newcommand{\id}{\operatorname{id}}
\newcommand{\supp}{\operatorname{supp}}
\newcommand{\wlim}{\mbox{ w-lim }}
\newcommand{\mymu}{{x_N^{-p_*}}}
\newcommand{\C}{{\mathbb C}}
\newcommand{\R}{{\mathbb R}}
\newcommand{\N}{{\mathbb N}}
\newcommand{\Z}{{\mathbb Z}}
\newcommand{\Q}{{\mathbb Q}}
\newcommand{\abs}[1]{\lvert#1\rvert}
\newtheorem{theorem}{Theorem}[section]
\newtheorem{corollary}[theorem]{Corollary}
\newtheorem{lemma}[theorem]{Lemma}
\newtheorem{definition}[theorem]{Definition}
\newtheorem{remark}[theorem]{Remark}
\newtheorem{proposition}[theorem]{Proposition}
\newtheorem{assertion}[theorem]{Assertion}
\newtheorem{problem}[theorem]{Problem}
\newtheorem{conjecture}[theorem]{Conjecture}
\newtheorem{question}[theorem]{Question}
\newtheorem{example}[theorem]{Example}
\newtheorem{Thm}[theorem]{Theorem}
\newtheorem{Lem}[theorem]{Lemma}
\newtheorem{Pro}[theorem]{Proposition}
\newtheorem{Def}[theorem]{Definition}
\newtheorem{Exa}[theorem]{Example}
\newtheorem{Exs}[theorem]{Examples}
\newtheorem{Rems}[theorem]{Remarks}
\newtheorem{Rem}[theorem]{Remark}

\newtheorem{Cor}[theorem]{Corollary}
\newtheorem{Conj}[theorem]{Conjecture}
\newtheorem{Prob}[theorem]{Problem}
\newtheorem{Ques}[theorem]{Question}
\newcommand{\pf}{\noindent \mbox{{\bf Proof}: }}
\newcommand{\Lag}{{\mathcal L}}

\renewcommand{\theequation}{\thesection.\arabic{equation}}
\catcode`@=11 \@addtoreset{equation}{section} \catcode`@=12

\title{On compactness in the Trudinger-Moser inequality}
\author{Adimurthi\\
 {\small TIFR CAM}\\
 {\small Sharadanagar, P.B. 6503}\\
 {\small Bangalore 560065, India}\\
{\small aditi@math.tifrbng.res.in}\\
\and Cyril Tintarev
\\
{\small Uppsala University}\thanks{Research done in part while visiting TIFR
Centre for Applicable
Mathematics in Bangalore and Chinese University of Hong Kong.}
\\
{\small P.O.Box 480}\\
{\small SE-751 06 Uppsala, Sweden}\\{\small
tintarev@math.uu.se}}

\date{}
\maketitle
\begin{abstract}
\begin{small}
We show that Moser functional 
$
J(u)=\int_\Omega \left(e^{4\pi u^2}-1\right)\dx, 
$ 
on the set
$\mathcal B=\{u\in H_0^1(\Omega):\; \|\nabla u\|_2\le 1\}$,
where $\Omega\subset\R^2$ is a bounded domain,
fails to be weakly continuous only in the following exceptional case. Define
$g_sw(r) = s^{-\frac12}w(r^s)$, $s>0$. 
If $u_k\rightharpoonup u$ in $\mathcal B$ while $\liminf J(u_k)>
J(u)$, then, with some $s_k\to 0$, 
$$u_k=g_{s_k}\left[(2\pi)^{-\frac12}\min\{1,\log\frac{1}{|x|}\}\right],$$
up to translations and up to a remainder vanishing in the Sobolev norm. 
In other words, the weak continuity fails only on translations of
concentrating Moser functions. 
The proof is based on a profile decomposition similar to that of Solimini
\cite{Solimini}, but with different concentration operators, pertinent to
the two-dimensional case.
\\[3mm]
\noindent  {\em 2010MSC:}\quad
Primary  \! 35J20, 35J35, 35J60; Secondary 46E35, 47J30, 58E05. \\[1mm]
\end{small}
\end{abstract}
\mysection{Introduction}
The purpose of this paper is to study weak continuity properties 
in the Trudinger-Moser inequality at the same
level of detail as the better understood
weak continuity properties of the critical nonlinearity in higher dimensions.  
We draw comparison between the Sobolev inequality that defines the continuous
imbedding
$\mathcal D^{1,p}(\R^N)\hookrightarrow L^{p^*}$, $p^*=\frac{pN}{N-p}$, when $N>p$, 
and the Trudinger-Moser inequality (see Yudovich \cite{Yudovich}, Peetre \cite{Peetre}, Pohozhaev
\cite{Pohozhaev}, Trudinger \cite{Trudinger} and Moser \cite{Moser}):

\begin{equation}
\label{TM}
\sup_{\mathcal B}\int_\Omega e^{\alpha_N
|u|^{N'}}\dx<\infty, \quad \mathcal B=\{u\in W_0^{1,N}(\Omega):\, \|\nabla
u\|_N\le 1\},
\end{equation}
where $\Omega\subset\R^N$ is a bounded domain,
$N'=\frac{N}{N-1}$, 
$\alpha_N=N\omega_{N-1}^{1/(N-1)}$ is the optimal constant (due to
 Moser \cite{Moser}), and $\omega_{N-1}$ is the area of the unit
$N-1$-dimensional sphere. 
Using the notation $\|u\|_q$ for the $L^q$-norm of $u$, we fix  
the norm of $W_0^{1,N}$ as $\|\nabla u\|_N$. 
A ball in $\R^N$ of radius $R$ centered at $y$
will be denoted $B_R(y)$, abbreviated to $B_R$ when $y=0$, and to $B$ if
$y=0$ and $R=1$. 
We will refer to the functional
\begin{equation}
\label{MoserFunc}
J(u)\eqdef\int_\Omega \left(e^{\alpha_N |u|^{N'}}-1\right)\dx
\end{equation}
as Moser functional. 
Analogs of Trudinger-Moser inequality has been established 
for more general Sobolev spaces by Adams (\cite{Adams},
the case of higher derivatives), and by Fusco, Lions and Sbordone
(\cite{FLS}, a weighted version of the Trudinger-Moser inequality which allows
nonlinearities of very high growth).
Imbeddings $\mathcal D^{1,p}(\R^N)\hookrightarrow
L^{p^*}$, $p<N$, and $W_0^{1,N}(B)\hookrightarrow\exp L^{N'}$, 
 defined, respectively, by Sobolev and Trudinger-Moser inequalities,
are optimal when the target space in the class of Orlicz spaces. 
Further refinement of these imbeddings is possible, however, in the larger class of rearrangement-invariant spaces,
where the correspondent Orlicz spaces can be identified on the scales of Lorentz, resp. Lorentz-Zygmund, spaces 
as $L^{p^*}=L^{p^*,p^*}$, resp. $\exp L^{N'}=L^{\infty,\infty;-1/N}$.  
For further details we refer the
reader to  Appendix~A.
\vskip4mm
It is well known that the critical Sobolev
nonlinearity $\int_{\R^N} |u|^{p^*}\dx$.
lacks weak continuity in $\mathcal
D^{1,p}(\R^N)$ at any point $u$ (consider any sequence of the form
$u_k(x)=u(x)+k^\frac{N-p}{p}w(kx)$, $w\neq 0$, $k\to\infty$, and apply Brezis-Lieb
lemma). In contrast to that, according to the result of Lions, Theorem~I.6 in 
(\cite{PLL2b}, Moser functional on $\mathcal B$ is weakly continuous at any point except zero, and it is also weakly continuous on every sequence in $\mathcal B$ 
that converges weakly to zero, unless $\|\nabla u_k\|_N\to 1$ and $u_k$ has exactly one point of concentration.

In restriction to radially symmetric functions, the result of Lions can be further refined by using
calculations from the paper \cite{Moser} of Moser, which allow to infer that
Moser functional on $\mathcal B$ lacks weak continuity only (up to
a remainder vanishing in $W_{0}^{1,N}$) on a sequence of Moser
functions \eqref{MF}, concentrating at the origin. We reproduce these
calculations in Appendix B.

This indicates that, even without the assumption of radial symmetry, the class of sequences on which the Moser functional fails to be weakly continuous, may be characterized not just by the mere property of concentration at one point, but by a specific asymptotic behavior. The main result of this paper, proved here for the case $N=p=2$, is that this class consists of the same sequences as in the radial case, but subjected to arbitrary translations. This is consistent with the observation made here that the concentrating sequences that do not vanish in the $\exp L^2$-norm are always asymptotically radial, in contrast with the case $N>2$. By a concentrating sequence we mean here a sequence bounded in the corresponding Sobolev space, convergent almost everywhere to zero, but which do not vanish in the $L^{p^*}$-norm for $N>p$, resp. in $\exp L^{N'}$-norm for $N=p$. In the  case $N>p$, any (generally nonradial) function $w$ can occur as a concentration profile, since the sequence $u_k(x)=k^{\frac{N-2}{2}}w(kx)$ will be a concentrating sequence, whose normalized deflations by the scale factor $k^{-1}$, that is, $k^{-\frac{N-2}{2}}u_k(k^{-1}x)$, equal $w$. In the case $N=p=2$,  the analogous sequence $u_k(x)=w(kx)$, vanishes in $\exp L^2$, i. e. it is not concentrating. The relevant counterpart of  the rescaling deflations, presented in this paper  for $N=2$,  forms sequences with discrete rotational symmetries whose rank goes to infinity,  forming in the limit radial concentration profiles.

In order to define concentration that describes, for sequences bounded in the $H^1_0$-norm, the defect of convergence in the $\exp L^2$-norm, we prove a suitable profile decomposition, similar to the decomposition in \cite{AOT} for the radial case. Our starting point (we do not survey here a vast earlier literature where profile decompositions are established under substantial additional assumptions, typically, for critical sequences for elliptic variational problems, mentioning only the pioneering work of Struwe \cite{Struwe84}) is the profile decomposition due to Solimini \cite{Solimini} (a similar
decomposition was independently proved by Gerard \cite{Gerard} and
extended to more general spaces by Jaffard \cite{Jaffard}, as well as by one of
the authors of this paper) expresses a subsequence of an arbitrary bounded
sequence in the Sobolev space $\mathcal D^{1,p}(\R^N)$, $N>p$,
as an asymptotic sum whose terms have the form $t_k^\frac{N-p}{p}w(t_kx)$, with
a remainder vanishing in $L^{p^*}$.
In the paper \cite{acc} (elaborated in \cite{ccbook}) existence of profile decomposition was proved, for the general Hilbert space equipped with a group (subject to some general conditions) of unitary operators. This in turn gave
rise to the notion of cocompact imbedding of Banach spaces $X\hookrightarrow Y$, which is, roughly speaking, amounts to the vanishing in the norm of $Y$ of the remainder of the profile decomposition in $X$. 
The functional-analytic profile decomposition prompted us to define operators that could play the role of Solimini's ``rescalings'', with the remainder in the profile decomposition vanishing in $\exp L^{N'}$. 
In \cite{AOT}, dealing with the subspace of radial functions of $W_0^{1,N}(B)$, pertinent rescalings are given by \eqref{gauge}.  In this paper we have adopted an extension of the unitary operators
\eqref{gauge} for the case $N=2$ to isomertires \eqref{ourD} on the whole space
$W_{0}^{1,N}(B)$. These isometries are no longer bijective. Furthermore, they are defined only for the integer value of the parameter, and, while the operators \eqref{gauge} form a group, the set of operators
\eqref{ourD} is only a semigroup. On the other hand, it is exactly the absence of
bijectivity (in fact, of surjectivity, since isometries are always injective) that is ultimately responsible for the radial profiles of concentration profiles.

The results of the paper are as follows.
In order to establish the structure of the exceptional sequences for Moser
functional, we employ a straightforward adaptation,
Theorem~\ref{abstractcc}, of the functional-analytic profile decomposition
theorem from \cite{ccbook}. Theorem~\ref{2cc-j} is an application of
Theorem~\ref{abstractcc}
to the Sobolev space $H_0^1(B)$ equipped with the semigroup \eqref{ourD}. In 
Theorem~\ref{cocompactness} we  verify
that the remainder of the profile decomposition vanishes in the $\exp L^2$-norm (which is an equivalent quasinorm
of $L^{\infty,\infty;-1/2}$, or in other words, that the imbedding $H_0^1(B)\hookrightarrow L^{\infty,\infty;-1/2}$ is
cocompact. Combining this and the optimal imbedding 
$H_0^1(B)\hookrightarrow L^{\infty,2;-1}$, one gets by the H\"older inequality 
that the imbedding $H_0^1(B)\hookrightarrow L^{\infty,q;-1/q-1/2}$ is cocompact for any $q>2$.
By analogy with Solimini's counterexample on p.333 of \cite{Solimini}, we also show
that the optimal imbedding  $H_0^1(B)\hookrightarrow L^{\infty,2;-1}$ is not cocompact.
The main result of the paper is as follows.

 \begin{theorem}
 \label{thm:weak_discontinuity} Let $\Omega\subset \R^2$ be a bounded domain,
 and let $J$ be the Moser functional \eqref{MoserFunc}.  
 If $u_k\in H_0^1(\Omega)$ is such that $\|\nabla u_k\|_ 2\le 1$, $u_k\rightharpoonup u$,  and 
 $\liminf J(u_k)> J(u)$, then 
there is a sequence $\zeta_k\in\bar\Omega$ and a sequence $s_k\in(0,1)$ such that $u_k-m_{s_k}(\cdot-\zeta_k)\to 0$  in the $H^1$-norm,
where \begin{equation}
\label{MF}
m_s(r)\eqdef(\omega_{N-1})^{-\frac{1}{N}}\log(1/s)^{\frac{1}{N'}}
\min\left\lbrace\frac{\log(1/r)}{\log(1/s)},1\right\rbrace, \quad r,s\in(0,1).
\end{equation}
\end{theorem}
The functions \eqref{MF} were used by Moser in \cite{Moser}  to prove optimality of the constant in \eqref{TM}, and are usually called Moser functions.

Profile decompositions, Theorem~\ref{abstractcc} and Theorem~\ref{2cc-j} are proved in Section 2. The proofs of Theorem~\ref{cocompactness} and
of Theorem~\ref{thm:weak_discontinuity} are given in Section 3. Appendix~A provides some background material on imbeddings of Sobolev spaces into Lorentz and Lorentz-Zygmund spaces and Appendix B summarizes properties of Moser functional in the radial case.
\vskip3mm
One of the authors (T.) thanks Michael Cwikel, Lubo{\v s} Pick and Yevgeniy
Pustylnik for discussions in connection to Appendix A, and Sergio Solimini for enlightening comments about profile decompositions.

\mysection{Profile decomposition in $H_0^1$}
We give below a definition of isometric dislocations, which extends   
the definition of dislocation operators from \cite{ccbook}
to the case of non-surjective isometries. 

\begin{definition}
\label{dislocations} Let $H_1$ be a separable infinite-dimensional
Hilbert space and let $H_0$ be its closed subspace. A set $D$ of isometric linear operators from $H_0$ to $H_1$ is a set of {\em isometric dislocations} if, whenever $u_k\in H$ and $g_k$, $h_k \in D$,
\begin{equation} \label{newiii}  g_ku_k\rightharpoonup 0 ,
h_kg_k^* \not\rightharpoonup 0
\Rightarrow \exists \{k_j\}\subset \N,
h_{k_j}u_{k_j}\rightharpoonup 0.
\end{equation}
One says that a sequence $u_k$ is $D$-weakly convergent to zero if for every sequence $g_k\in D$, $g_ku_k\rightharpoonup 0$.
\end{definition}
Note that we deviate in this
section from the notations in \cite{ccbook} by interchanging the operator set
$D$ and the set of adjoints $D^*=\{g^*: g\in D\}$. 
This interchange is important for coherence with the applications in this paper,
while it is of no significance for the applications studied in \cite{ccbook},
or for most typical applications elsewhere when $D$ is a group of
unitary operators, and therefore $D^*=D$.

\begin{theorem}
\label{abstractcc}
Let $H_1$ be a separable infinite-dimensional Hilbert space with a closed subspace
$H_0$ and a set of isometric dislocations $D:H_0\to H_1$.
If $u_{k}\in H_0$ is a bounded sequence and $u_k\rightharpoonup 0$, then there
exists a set $\N_0\subset\N$, $w^{(n)}  \in
H$, $g_{k} ^{(n)} \in D$, $g_{k} ^{(1)} =id$, with $k \in\N$ and $n\in\N_0$,
such that for a renumbered subsequence,
\begin{eqnarray}
\label{w_n} &&w^{(n)}=\wlim  {g_{k} ^{(n)}}u_k,
\\
\label{separates} &&{g_{k} ^{(n)}} {g_{k}
^{(m)}} ^{*} \rightharpoonup 0  \mbox{ for } n \neq m,
\\
\label{norms} &&\sum_{n \in \N_0} \|w ^{(n)}\|^2  \le
\limsup \|u_k\|^2,
\\
\label{BBasymptotics} &&u_{k}  -  \sum_{n\in\N_0}  {g_{k} ^{(n)}}^*
w^{(n)}  \cw 0,
\end{eqnarray}
where the series $\sum_{n\in\N_0}  {g_{k} ^{(n)}}^* w^{(n)}$ converges
uniformly in $k$.
\end{theorem}

\begin{proof}
The proof is an elementary modification of the proof of Theorem~3.1 from
\cite{ccbook} and we give it in an abbreviated form. 

1. One shows first that (\ref{norms}) follows from (\ref{w_n})
and (\ref{separates}). The proof of this step is analogous to that in \cite{ccbook} and can be omitted.
\par
2. Observe that if $u_k\cw 0$, the theorem is verified with
$\N_0=\emptyset$. If not so, consider the expressions of
the form
\begin{equation}
\label{def_w1}
w^{(1)}=:\wlim {g_k^{(1)}}u_k.
\end{equation}
Since we assume that $u_k$ does not converge
$D$-weakly to zero, there exists necessarily a renumbered sequence
$g_k^{(1)}$ that yields a non-zero limit in (\ref{def_w1}).
\par
Let
\begin{equation*}
v_k^{(1)}= u_k - {g_k^{(1)}}^* w^{(1)},
\end{equation*}
and observe, by (\ref{def_w1}), that
\be \label{g1v1} 
{g_k^{(1)}}v_k^{(1)}={g_k^{(1)}}u_k-
w^{(1)}\rightharpoonup 0.\ee
If $v_k^{(1)} \cw 0$, the theorem is verified with $\N_0=\{1\}$.
If not -- we repeat the argument above -- there exist,
necessarily, a sequence $g_k^{(2)}\in D$ and a $w^{(2)}\neq 0$
such that, on a renumbered subsequence,
\begin{equation*}
g_k^{(2)}
v_k^{(1)} \rightharpoonup w^{(2)}.
\end{equation*}
Let us set
\begin{equation*}
v_k^{(2)}= v_k^{(1)} - {g_k^{(2)}}^* w^{(2)}.
\end{equation*}
Then we will have an obvious analog of (\ref{g1v1}):
\begin{equation}
\label{w2-1}  {g_k^{(2)}} v_k^{(2)}=
  {g_k^{(2)}}v_k^{(1)}  -
 w^{(2)} \rightharpoonup 0.
\end{equation}
If we assume that
\begin{equation*}
{g_k^{(1)}}{g_k^{(2)}}^*\not\rightharpoonup 0,
\end{equation*}
then by (\ref{newiii}), (\ref{w2-1}),
\begin{equation*}
{g_k^{(1)}} (v_k^{(1)} - {g_k^{(2)}}^*w^{(2)})\rightharpoonup 0,
\end{equation*}
which, due to (\ref{g1v1}), yields
\begin{equation}
\label{w2-4}{g_k^{(1)}} {g_k^{(2)}}^*w^{(2)}\rightharpoonup 0.
\end{equation}
We now use (\ref{newiii}) again to replace in (\ref{w2-4})
$ {g_k^{(1)}}$ with
$ {g_k^{(2)}}$, which results in
\be w^{(2)}\rightharpoonup 0,\ee which cannot be true since we
assumed $w^{(2)}\neq 0$. This contradiction implies that
\begin{equation*}
{g_k^{(1)}} {g_k^{(2)}}^* \rightharpoonup 0.
\end{equation*}
Since for bounded sequences of operators $A_k\rightharpoonup 0$
implies $A_k^*\rightharpoonup 0$, we also have
\begin{equation*}
{g_k^{(2)}} {g_k^{(1)}}^* \rightharpoonup 0.
\end{equation*}
Recursively we define:
\begin{equation}
\label{def:next_v} v_k^{(n)}: = v_k^{(n-1)}-{g_k^{(n)}}^* w^{(n)} =
u_k - {g_k^{(1)}}^* w^{(1)} - \dots - {g_k^{(n)}}^* w^{(n)},
\end{equation}
where
\begin{equation*}
w^{(n)}=\wlim
{g_k^{(n)}} v_k^{(n-1)},
\end{equation*}
calculated on a successively renumbered subsequence. We
subordinate the choice of $g_k^{(n)}$ and thus extraction of this
subsequence for every given $n$ to the following requirements. For
every $ n\in\N$ we set
\begin{equation*}
W_n=\{w\in H_1\setminus\{0\}:\;\exists g_j\in D,
\{k_j\}\subset\N: g_jv_{k_j}^{(n)}\rightharpoonup
w\},
\end{equation*}
and
\begin{equation*}
t_n=\sup_{w\in W_n}\|w\|.
\end{equation*}
 If for some $n$, $t_n=0$, the theorem is proved. Otherwise, like in \cite{ccbook} we choose a
$w^{(n+1)}\in W_n$ such that
\begin{equation}
\label{w-t} \|w^{(n+1)}\|\ge\frac12 t_n
\end{equation}
and the sequence $g_k^{(n+1)}$ is chosen so that on a subsequence
that we renumber,
\begin{equation}
\label{next_w}
{g_k^{(n+1)}}v_k^{(n)}
\rightharpoonup w^{(n+1)}.
\end{equation}
An argument analogous to the one brought above for $n=1,2$ shows
that
\begin{equation}
\label{separation_p_q} g_k^{(p)}{g_k^{(q)}}^*\rightharpoonup 0 \mbox{
whenever } p \neq q, p,q\leq n.
\end{equation}
This allows to deduce immediately (\ref{w_n}) from (\ref{next_w}).

3. Similarly to \cite{ccbook} one derives that $t_n\to 0$ from which
subsequently follows the asymptotics \eqref{BBasymptotics}. The convergence of
the series \eqref{BBasymptotics}, like in \cite{ccbook}, is a modification of
Plancherel formula that requires to extract a sufficiently rarefied subsequence
of $u_k$ to assure sufficient approximation of orthogonality by the
asymptotically orthogonal terms ${g_k^{(n)}}^*w^{(n)}$. 
\end{proof}

From now on we assume, without loss of generality that $\Omega\subset
B_\frac12$. This restriction is not substantial and
can be removed by linear rescaling, since, if we denote the Moser function
subordinated to an annulus $t<r<R$ as $m_t^{(R)}$, an easy computation shows
that, for any $R>0$,
\[
\lim_{t\to 0}\|\nabla(m_t^{(R)}-m_t)\|_2\to 0.
\]

Let us now specify $H_1$, $H_0$  and $D$ as follows: $H_1=H^1_0(B)$, $H_0=H_0^1(\Omega)$, and

\begin{equation}
\label{ourD}
D=\{g_{j,\zeta}u(z)=j^{-\frac12}u(\zeta+z^j), \zeta\in \bar\O, j\in\N\}.
\end{equation}
Here and in what follows, the expression $z^j$ stays for the power of the
complex number representing a point in $z\in\R^2$, and translations of
functions $H^1_0(\Omega)$ are understood, using extension by zero, as elements
of $H^1_0(B)$. 

We have the following obvious property of the asymptotic profiles \eqref{w_n}.
\begin{remark}
If the sequences $u_k\in H_0^1(\Omega)$, $j_k\in\N$, and $\zeta_k\in
\Omega$ are such that $g_{j_k,\zeta_k}u_k\rightharpoonup w$, $j_k\to\infty$
and $z_k\to z_0$, then $w$ is radially symmetric.   
\end{remark}
There is also an obvious analytic form  of \eqref{separates}.
\begin{lemma}
Let $D$ be the set of operators as above. Two sequences,
$\{g_{j_k^{(1)},z_k^{(1)}}\}_k\subset D$ and
$\{g_{j_k^{(2)},z_k^{(2)}}\}_k\subset D$, with $j_k^{(1)}\to\infty$ and
$j_k^{(2)}\to\infty$, satisfy 
\[
g_{{j_k^{(2)},z_k^{(2)}}}{g^*_{{j_k^{(1)},z_k^{(1)}}}}\rightharpoonup 0
\]
if and only if 
\begin{equation}
\label{asymptort}
\inf_k|z_k^{(2)}-z_k^{(1)}|>0 \mbox{ or } |\log j_k^{(2)}-\log j_k^{(1)}|\to
\infty.
\end{equation}
\end{lemma}

This allows to express Theorem~\ref{abstractcc} for our particular choice of $H_1$, $H_0$ and $D$ as follows.

\begin{theorem}
\label{2cc-j} Let $\Omega\subset B_\frac12\subset\R^2$ and let
$u_k\rightharpoonup 0$ be a sequence  in $H^{1}_{0}(\Omega)$. There
exist $j_k^{(n)}\in \N$, $\lim_{k\to\infty} j_k^{(n)}=\infty$,  and
$z_k^{(n)}\in \bar \Omega$, 
$\lim_{k\to\infty} z_k^{(n)}= z_n\in\bar\Omega$, $k \in\N$, $n\in\N$, 
such that  for a renumbered subsequence,
\begin{eqnarray}
\label{w_n=j} &&w^{(n)}(|z|)=\wlim \left({j_k^{(n)}}\right)^{-1/2}u_k(z_k^{(n)}+z^{j_k^{(n)}}),
\\
\label{separates=j} &&    z_m\neq z_n \mbox{ or } |\log j_k^{(m)}-\log j_k^{(n)}|\to \infty  \mbox{ whenever } n \neq m,
\\
\label{norms=j} &&\sum_{n \in \N} \int_B|\nabla w ^{(n)}|^2\dx  \le \limsup \int_\Omega|\nabla u_k|^2\dx,
\\
\label{BBasymptotics=j} &&u_{k}  -  \sum_{n\in\N}  {{j_k^{(n)}}}^{1/2}w^{(n)}(|z-z_n|^{1/j_k^{(n)}}) \cw 0,
\end{eqnarray}
and the series  $\sum_{n\in\N}  {{j_k^{(n)}}}^{1/2}w^{(n)}(|z-z_n|^{1/j_k^{(n)}})$ converges in $H^1_{0}(B)$
uniformly in $k$.
\end{theorem}

We note also that for any radially symmetric function $w\in H_0^1(B))$, the
sequence ${g_{j,\zeta}}^*w=j^\frac12w(|z-\zeta|^j)$ is dependent on $\zeta$
continuously in $H^1_0(B)$ and uniformly in $j\in\N$, so in the asymptotic
expansion \eqref{BBasymptotics=j} we could replace 
${g^*_{j_k^{(n)},z_k^{(n)}}}w^{(n)}$ by 
${g^*_{j_k^{(n)},z_n}}w^{(n)}$.

We complement this profile decomposition by the statement below, which
identifies the convergence of the remainder in \eqref{BBasymptotics=j} in
as convergence in the Banach space $\exp L^2$.

\begin{theorem}
\label{cocompactness}
Let $\Omega\subset B_\frac12$, and let $D$ be the set \eqref{ourD}. 
If a sequence $u_k\in H_0^1(\O)$ is $D$-weakly
convergent to zero, then $u_k\to 0$ in $\exp
L^2$. In particular, for any $\lambda>0$,
$\int_\Omega \left(e^{\lambda u_k^2}
-1\right)\dx\to 0$.
\end{theorem}
Note that the restriction $\Omega\subset B_\frac12$ is not substantial and this statement can be restated for any bounded domain by linear rescaling. 

Before we prove the theorem, we state a corollary and a counterexample. 
\begin{corollary} 
\label{cor:cocomp}
Let $\Omega\subset B_\frac12$. If a sequence $u_k\in H_0^1(\Omega)$ is
$D$-weakly convergent to zero, then, for every
$q\in (2,\infty]$, it
converges to zero in the Lorentz-Zygmund space $L^{\infty,q;-1/q-1/2}$ .
\end{corollary}
This easily follows from the embedding
$H_0^1(B)\hookrightarrow L^{\infty,2;-1}$ and the interpolation by H\"older
inequality between $L^{\infty,2;-1}$ and $L^{\infty,\infty;-1/2}=\exp L^2$. 

\begin{remark}
\label{counterexample}
By analogy with the counterexample given by Solimini \cite{Solimini} that
remainder in his profile decomposition does not necessarily converge in the
sense of $L^{p^*,p}$,  we can show that Corollary~\ref{cor:cocomp} does not
extend to the endpoint case $q=2$. Our construction of the sequence is
analogous to Solimini's.

Let $v\in C_0^\infty((e^{-3},e^{-2}))$,  $v\neq 0$, let $w(x)=v(|x|)$ and let 
\begin{equation}
\label{w-k} 
w_k=k^{-1/2}\sum_{j=1}^kg_{2^j,0}w,\;k\in\N.
\end{equation}
Let us show that, for an arbitrary sequence $j_k\to\infty$ and $\zeta_k\in \Omega$, one has
$g_{j_k,\zeta_k} w_k\rightharpoonup 0$. 

By the standard density argument,  it suffices to prove that  
 $\int g_{j_k,\zeta_k} w_k\psi\to 0$ for each $\psi\in L^2(B)$.
Indeed, since the supports for the individual terms in the sum defining
$w_k$ remain disjoint under the action of $g_{j_k,\zeta_k}$,
\[
\left(\int g_{j_k,\zeta_k} w_k\psi\right)^2\le \|g_{j_k,\zeta_k} w_k\|_2\|\psi\|_2,
\] 
and an elementary computation shows that 
\[
\|g_{j_k,\zeta_k} w_k\|_2\le \|w_k\|^2_2\le k^{-1}\|w\|^2_2\to 0. 
\]
Observe now  that the terms in the sum in \eqref{w-k} have disjoint supports, which implies
that $\|\nabla w_k\|_2=\|\nabla w\|_2$.
Therefore we have a bounded sequence in $H_0^1(B)$ which converges $D$-weakly to zero.
However, an analogous calculation also gives that
\begin{equation}
\int\frac{w_k^2}{r^2(\log\frac{1}{r})^2}=\int\frac{w^2}{r^2(\log\frac{1}{r})^2}. 
\end{equation}
Note that $\frac{1}{r^2(\log\frac{1}{r})^2}$ is a decreasing function on the support of $w_k$, which implies that 
\begin{equation}
\int\frac{w^2}{r^2(\log\frac{1}{r})^2}\le \int\frac{{w_k^\star}^2}{r^2(\log\frac{1}{r})^2}\le C\|w_k\|^2_{L^{\infty,2;-1/2}},
\end{equation}
and thus $w_k$ does not converge to zero in $L^{\infty,2;-1/2}$.  
\end{remark}
\mysection{Proofs of Theorem~\ref{cocompactness} and Theorem~\ref{thm:weak_discontinuity}}

The proof of Theorem~\ref{cocompactness} is based on the following five lemmas.
None of the lemmas, except, possibly, Lemma~\ref{ArOK}, is a new result,
but we have included them for the sake of completeness of the presentation.
We recall that the expression $z^j$, $j\in\N$, refers to a power of
the complex number $z$ representing a point in $\R^2$, the set of operators $D$
is defined by \eqref{ourD}, and $D$-weak convergence is defined in
Definition~\ref{dislocations}.

We start with the following elementary statement.
\begin{lemma}
\label{o2b} 
Let $\bar\Omega\subset B$. If $u_k\in H_0^1(\Omega)$ 
and $u_k\cw 0$, then  $g_{j_k,\zeta_k}u_k\rightharpoonup 0$ whenever
$j_k\in \N$, $\zeta_k\in
B$.
\end{lemma}
\begin{proof}
For $\zeta_k\in \bar\Omega$
the assertion follows directly from
the definition of $D$-weak convergence. 
For $\zeta_k\in B\setminus\bar\Omega$, note that $u_k\cw 0$ implies
$u_k\rightharpoonup 0$
and that operators $g_{j_k,\zeta_k}$ map any sequence $u_k\rightharpoonup 0$,
such that
$\inf_k\mathrm{dist}(\zeta_k,\supp u_k)>0$ to a sequence that weakly converges
to zero. 
Finally, the case $\mathrm{dist}(\zeta_k,\supp u_k)\to 0$ can be easily reduced
by a continuity argument to
the case $\zeta_k\in\bar\Omega$.   
\end{proof}
In what follows, the two-dimensional Lebesgue measure will be denoted 
${\mathrm d} x {\mathrm d} y$ when the integration variable is  $z$, and ${\mathrm d}\xi {\mathrm d}\eta$  when the integration variable is called $\zeta$.  Let us introduce the averaging operator
$$A_ru(z)=\fint_{B_r(z)}u(\zeta)\ud\xi\ud\eta =\frac{1}{|B_r(z)|}
\int_{B_r(z)}u(\zeta)\ud\xi\ud\eta.$$
\begin{lemma}
\label{averaging}
Let $u_k\in H_0^1(\Omega)$, $\bar\Omega\subset B$. 
If $u_k\cw 0$ and $r_k\downarrow 0$, then $A_{r_k}u_k\cw 0$.  
\end{lemma}
\begin{proof}
Without loss of generality, assume that $u_k\ge 0$.
It suffices to verify that for each nonnegative $\varphi\in C_0^1(B)$ and
for each sequence $j_k\in\N$, $\zeta_k\in\bar\Omega$,
$$
\int_B \varphi(z) g_{j_k,\zeta_k}A_{r_k}u_k(z)\ud x\ud y\to 0.
$$
Then we have
\begin{equation}
 \label{av1}
\begin{split}
\int_B \varphi(z) g_{j_k,\zeta_k}
\left(\fint_{B_{r_k}}u_k(z+\zeta)\ud\xi\ud\eta\right)
\ud x\ud y
\\
=\int_B j_k^{-\frac12} \fint_{B_{r_k}}
u_k(\zeta+\zeta_k+z^{j_k})\ud\xi\ud\eta \varphi(z)\ud x\ud y
\\
\le   
C\sup_{\zeta\in B} \int_B
j_k^{-\frac12}  u_k(z^{j_k}+\zeta) \ud x\ud y.
\end{split}
\end{equation}
With suitable $\zeta_k'\in B$ one can estimate the last expression by
$$
\int_B j_k^{-\frac12}  u_k(z^{j_k}+\zeta_k') \ud x\ud y=
\int_B g_{j_k,\zeta_k'}u_k(z) \ud x\ud y.
$$
Since $u_k\cw 0$,  the right hand side converges to zero by Lemma~\ref{o2b}.  
\end{proof}
\begin{lemma} 
\label{lem:aver-cont} Let $\bar\Omega\subset B$.
There exists $C>0$ such that for every $w\in L^2(\Omega)$, extended by
zero to
$\R^2$, and for every small $r>0$,
\begin{equation}
 \label{A_r-modulus}
|A_rw(z')|\ge |A_rw(z)|- C\|w\|_2\frac{|z-z'|^\frac12}{r^\frac32},\,
z,z'\in B.
\end{equation}
\end{lemma}
\begin{proof}
From the definition of the averaging operator $A_r$, by the Cauchy
inequality, and denoting symmetric difference of
sets, $(A\cup B)\setminus(A\cap B)$ as $A\vartriangle B$,we have
\begin{equation*}
\begin{split}
|A_rw(z')-A_rw(z)|\le 
Cr^{-2}\int_{B_r(z')\vartriangle B_r(z)} |w|\ud x\ud y
\le
Cr^{-2}\|w\|_2\sqrt{|B_r(z')\vartriangle B_r(z)|}
\\
\le  C \|w\|_2r^{-2}\sqrt{r|z-z'|}= C
\|w\|_2\frac{|z-z'|^\frac12}{r^\frac32} 
\end{split}
\end{equation*}
from which \eqref{A_r-modulus} is immediate.
\end{proof}
In what follows $w^\star$ denotes the symmetric decreasing rearrangement of $w$.
\begin{lemma}
\label{cocorad}
Let $u_k$ be a bounded sequence in $H_0^1(\Omega)$. 
If for every $j_k\in\N$, the sequence $g_{j_k,0}u^\star_k$ converges to zero in
measure, then $u_k\to 0$ in $\exp L^2$.
\end{lemma}
\begin{proof}
Since convergence in measure, for sequences bounded in $H_0^1(B)$, implies
weak convergence in $H_0^1(B)$, it follows from identity \eqref{g_mu} that 
\begin{equation}
\label{mj} 
\int_B\nabla m_{e^{-j_k}}\cdot \nabla
u_k^\star
=
\int_B\nabla m_{e^{-1}} \cdot \nabla g_{j_k,0}u_k^\star\to 0,
\end{equation}
The set of isometries $\{g_{j,0}\}_{j\in\N}$, once their domain is restricted to
the radial subspace $H_{0,r}^1(B)$, becomes a subset of the 
multiplicative group of isometries $\{h_s\}_{s>0}$ defined by \eqref{gauge}.
This group has the following, easily verifiable, property:
if $s_k$ is a bounded sequence and $v_k\rightharpoonup 0$ in $H_{0,r}^1(B)$,
then $h_{s_k}v_k\rightharpoonup 0$. Then it follows from \eqref{mj}
that for any sequence $r_k\in(0,e^{-1})$
(with $j_k\in\N$ chosen so that $0\le \log\frac{1}{r_k}-j_k\le 1$), 
\begin{equation}
 \int_B\nabla m_{r_k}\cdot \nabla u_k^\star\to 0.
\end{equation}
Then taking into account \eqref{sp0} and \eqref{sp}, we conclude that 
$$\sup_{r\in(0,e^{-1})}\frac{u_k^\star(r)}{(\log \frac{1}{r})^\frac{1}{2}}\to
0.$$
Moreover, by the compactness in the one-dimensional Morrey imbedding, 
we also have 
$$\sup_{r\in[e^{-1},1)}\frac{u_k^\star(r)}{(1+\log \frac{1}{r})^\frac{1}{2}}\to
0.$$
Combining two last relations we arrive at
$$\sup_{r\in(0,1)}\frac{u_k^\star(r)}{(1+\log \frac{1}{r})^\frac{1}{2}}\to 0,$$
that is, $u_k\to 0$ in $L^{\infty,2;-1}=\exp L^{2}$.
\end{proof}

\begin{lemma}
\label{ArOK}
Assume that $u_k\in H_0^1(\Omega)$, $\|\nabla u_k\|_2\le 1$, and $\|u_k\|_{\exp
L^2}\not\to 0$. Then, for a renumbered subsequence, there exists a sequence
$j_k\in\N$, $\zeta_k\in\bar\Omega$, such that for every $\epsilon>0$ there exits
$\rho>0$ such that 
\begin{equation}
\label{eq:ArOK}
j_k^{-\frac12} |A_{\rho^{j_k}}u_k(\zeta_k)|\ge \epsilon.
\end{equation}
\end{lemma}
\begin{proof} 
Since $u_k$ does not converge to zero in $\exp L^2$, by
Lemma~\ref{cocorad}, there is a renumbered subsequence such that, for any
$\epsilon>0$, the
measure of the sets
\begin{equation}
M_k^\star=\{z\in B:\,j_k^{-\frac12}u^\star_k(r^{j_k})\ge 2\epsilon\} 
\end{equation}
is bounded away from zero. 
Since $u^\star$ is a decreasing function,  there exists $\rho>0$ such that 
$j_k^{-\frac12}u^\star_k(r)\ge 2\epsilon$ for all $r\le \rho^{j_k}$.
This immediately implies that
\begin{equation}
\label{Mkg}
|M_k|\ge \pi\rho^{2j_k}, \text{ where } M_k=\{z\in
\Omega:\,j_k^{-\frac12}|u_k(z)|\ge 2\epsilon\}. 
\end{equation}
Let us now use a well-known inequality (see e.g. inequality (4) in
\cite{Solimini}) that holds for every $u\in H_0^1(B)$:
\begin{equation}
\label{osc}
\|A_ru-u\|_2\le C r\|\nabla u\|_2.
\end{equation}
In particular, we have
\begin{equation}
\label{osc2}
\int_{M_k}|A_{r_k}u_k-u_k|^2\le C r_k^2,
\end{equation}
which, combined with \eqref{Mkg}, gives
\begin{equation}
C r_k^2\ge  \pi\rho^{2j_k}\inf_{z\in M_k}|A_{r_k}u_k(z)-u_k(z)|^2,
\end{equation}
from which we conclude that there exists a sequence $\zeta_k\in M_k$,
such that
\begin{equation}
j_k^{-\frac12}|A_{r_k}u_k(\zeta_k)|\ge
2\epsilon-Cj_k^{-\frac12}\frac{r_k}{\rho^{j_k}}\ge
\epsilon,
\end{equation}
from which the assertion of the lemma is immediate once we choose
$r_k=\rho^{j_k}$. 
\end{proof}

{\em Proof of Theorem~\ref{cocompactness}.}
Assume that there exists a sequence $u_k\in H_0^1(\Omega)$, $u_k\cw 0$, which
does not converge to zero in $\exp L^2$. Then let $\rho>0$, $j_k$ and $\zeta_k$ be as in Lemma~\ref{ArOK}.

Let us fix $\epsilon>0$ and evaluate
the measure of the sets
\[
N_k=\{z\in B:\,|g_{j_k,\zeta_k}A_{\rho^{j_k}}u_k(z)|\ge\epsilon/2\}.
\]
 Applying
Lemma~\ref{lem:aver-cont} with $w=j_k^{-\frac12}u_k$,
$r=\rho^{j_k}$ and $z'=\zeta_k$, we have from \eqref{A_r-modulus} and
\ref{eq:ArOK}, for all $z\in B$ such that $|z-\zeta_k|\le \rho^{5 j_k}$,
\begin{equation}
j_k^{-\frac12}|A_{\rho^{j_k}}u_k(z)|\ge \epsilon- C\rho^{j_k}\ge \epsilon/2, 
\end{equation}
for all $k$ sufficiently large. 
Then, from the definition of the set $N_k$ above and the definition
of $g_{j,\zeta}$ in \eqref{ourD}, it follows that the set
$N_k$ contains the ball 
\[
\{z\in B:\,|z|^{j_k}\le \rho^{5 j_k}\},
\]
that is, $N_k\supset B_{\rho^5}$.

We conclude that $g_{j_k,z_k}A_{\rho^{j_k}}u_k$ does not converge to zero in
measure, and thus, $A_{\rho^{j_k}} u_k$ does not converge to zero $D$-weakly.
Then by Lemma~\ref{averaging}, $u_k$ does not converge  to zero $D$-weakly,
which contradicts the assumption of the theorem. 
\hfill\qed

{\em Proof of Theorem~\ref{thm:weak_discontinuity}}
Since Moser functional is lower weakly semicontinuous by Fatou lemma, and since
it is known that it lacks weak continuity on the unit ball of $H_0^1(B)$  only
at zero, we may assume without loss of generality that $u_k\in H_0^1(\O)$
is such that $\|\nabla u_k\|_2\le 1$, $u_k\rightharpoonup 0$ and
$\lim J(u_k)> J(0)=0$. 
Consider a renumbered subsequence of $u_k$ satisfying \eqref{BBasymptotics=j} and recall that the remainder there converges to zero in $\exp L^2$ by Theorem~\ref{cocompactness}.
Then, leaving details of separating supports to the reader, we have
\[
  J(u_k)=\sum_{n\in\mathbb N} J( v_k^{(n)}) + o(1).
\]
where 
\[
 v_k^{(n)}={{j_k^{(n)}}}^{1/2}w^{(n)}(|z-z_n|^{1/j_k^{(n)}})
\]

Note that since the argument of $J$ converges weakly to zero and $\|\nabla v_k^{(n)}\|_2=\|\nabla w^{(n)}\|_2$, one has $J(v_k^{(n)})\to 0$ whenever $\|\nabla w^{(n)}\|_2<1$. 
By assumption, $\lim J(u_k)>0$, which implies that for at least one value of $n$,  $\|\nabla w^{(n)}\|_2\ge 1$. Without loss of generality assume that this value of $n$ is $1$.
Comparing this with \eqref{norms=j}, we conclude that $w^{(n)}=0$ whenever $n>1$. 
Therefore 
\[
u_k = v_k^{(1)}+\omega_k
\]
with $\omega_k\cw 0$ and in particular, $\omega_k$ vanishes in $\exp L^2$. 

Note now, that since $\omega_k\cw 0$, one has $(\omega_k,v_k^{(1)})\to 0$, and, consequently,
\[
1\ge\limsup \|\nabla u_k\|_2^2=\limsup \|\nabla v_k^{(1)}+\nabla \omega_k\|_2^2=1+\limsup \|\nabla \omega_k\|_2^2.
\]
Therefore $\omega_k\to 0$ in $H^1$. 

If $w^{(1)}$ is not a Moser function, then by Proposition~\ref{weakrad}, $ J(v_k^{(1)})\to 0$. A slight modification of the proof of Proposition~\ref{weakrad} gives also $J(u_k)\to 0$, which is a contradiction.
Consequently, using \eqref{g_mu}, we have $u_k-m_{s_k}\to 0$ in $H^1$ with some sequence  $s_k\to 0$, which proves the theorem.
\hfill\qed

\section*{Appendix A. Optimal imbeddings of Sobolev spaces into Lorentz-Zygmund spaces}
Let $f:\, \R^N \to \mathbb R$ be a measurable function. The distribution function $\alpha_f$ and a
nonincreasing  rearrangement $f^\star$ of $f$ are defined as follows:
\[
\alpha_f(s)=|\{x \in \R^N;\, |f(x)|>s\}| \,\, \mbox{ and }\,\, f^*(t)=\inf\{s>0;\, \alpha_f(s) \leq t\}.
\]
Lorentz-Zygmund spaces $L^{p,q;\alpha}$, introduced by Bennett and Rudnick \cite{BR}, 
is a family of spaces that contains the both the Lorentz spaces $L^{p,q;0}=L^{p,q}$ and the Zygmund spaces 
$L^{0,\infty;-\alpha}=Z^\alpha$. They are defined as spaces of all measurable functions on a unit ball with bounded quasinorms
\begin{equation}
\label{LZ} 
\begin{split}
\|u\|_{p,q;\alpha}=
\left(\int_0^1\left[t^{1/p}(\log\frac{e}{t})^\alpha f^\star(t)\right]^q \frac{\mathrm d t}{t}\right)^\frac{1}{q},\,q\in(0,\infty)
\\
\|u\|_{p,\infty;\alpha}= 
\sup_{t\in(0,1)} \left|t^{1/p}(\log\frac{e}{t})^\alpha f^\star(t)\right|,\,q=\infty.  
\end{split}
\end{equation}
For the purpose of this paper we consider the range $p\in(1,\infty)$.
The definition of the Lorentz space $L^{p,q}(\R^N)$ is the same as the
definition of $L^{p,q;0}$ above with the domain of integration
$t\in(0,\infty)$ instead of $t\in (0,1)$. 

Lorentz space $L^{r,r}$ is equivalent to the Lebesgue space $L^r$, and
Lorentz space $L^{r,\infty}$ is equivalent to the Marcinkiewicz space
$M^r$, also known as the weak-$L^r$ space. 
Lorentz-Zygmund space $L^{\infty,\infty,-1/N'}$ is equivalent to
the Orlicz space $\exp L^{N'}$ of the Moser functional. 
 
For more background material on Lorentz and Lorentz-Zygmund spaces we
refer the reader to Bennett and Sharpley \cite{BS} and Bennett and Rudnick
\cite{BR}.

The reason why the domain of the functions considered here is a unit ball,
rather than $\R^N$, lies 
in the role of Lorentz-Zygmund spaces in the imbeddings of Sobolev spaces with the gradient norm. 
Completion $\mathcal D^{1,N}(\R^N)$ of the normed space $C_0^\infty(\R^N)$
equipped with the norm $\|\nabla u\|_N$ on does not
admit a continuous imbedding even into the space of distributions, which means
that the space $\mathcal D^{1,N}(\R^N)$ cannot be consistently defined as
a function space. On the other hand, Friedrichs' inequality gives that
the completion of $C_0^\infty(B)$ in the same gradient norm $\|\nabla u\|_N$ is
continuously imbedded into $L^N(B)$, defining the space $W^{1,N}(B)$ with
the equivalent Sobolev norm $\|\nabla u\|_N$.  
It should be also noted that  $\|\nabla u\|_N$ also expresses the gradient norm 
of the Laplace-Beltrami operator on the hyperbolic space $\mathbb H^N$ when
written under the coordinate map of Poincar\'e ball, which allows to understand
the ``Euclidean'' Sobolev space $W_0^{1,N}(B)$ of the unit ball in $\R^N$ is
isometric to the Sobolev space $\dot W^{1,N}(\mathbb H^N)$ of a complete
non-compact Riemannian manifold, giving the
unit ball in $\R^N$ when $p=N$ an intuitively equal standing, when Sobolev
spaces are concerned, with the whole $\R^N$ when $p<N$.

In fact, the similarities between $\mathcal D^{1,p}(\R^N)$ for $N>p$, and
$W_0^{1,N}(B)$, equipped with the gradient norm, are quite extensive. 
In particular, while the norm $\mathcal D^{1,p}(\R^N)$ remains invariant under dilations
$u\mapsto t^\frac{N-p}{p}u(t\cdot)$ also in the case $p=N$ 
(even if $\mathcal D^{1,N}(\R^N)$ is no longer a functional space), the subspace of
radial functions of $W_0^{1,N}(B)$ admits a different isometry group \eqref{gauge} 
of nonlinear dilations. 
There are also similarities in imbeddings of Sobolev type into
rearrangement-invariant spaces. While the standard 
limiting Sobolev inequality and the Trudinger-Moser inequality are quite different in appearance, 
this difference finds its explanation when one considers imbeddings of Sobolev spaces 
into the scales of correspondent Lorentz or Lorentz-Zygmund spaces.

We observe first that there is a continuous imbedding of $\mathcal D^{1,p}(\R^N)$, $N>p$, into 
$L^{p^*,p}$, which is immediate from the Hardy inequality for $u^\star$, 
combined with the Polia-Szeg\"o inequality. The analogous imbedding for $p=N$ is
$W^{1,N}_0(B)\hookrightarrow L^{\infty,N;-1}$, based on the inequality of Hardy type 
\[
\int_B|\nabla u|^N\ge C_N\int_B\frac{|u|^N}{\left(r\log \frac{1}{r}\right)^N}, u\in C_0^\infty(B).
\]
(see Adimurthi and Sandeep \cite{AS} and Adimurthi and Sekar \cite{ASec}; 
in the case $N=2$ it was proved first by Leray \cite{Leray}). 

Lorentz spaces are nested with respect to the second index, and thus there is a continuous imbedding 
$\mathcal D^{1,p}(\R^N)\hookrightarrow L^{p^*,q}$, $N>p$, for all $q\in[p,\infty]$ 

A continuous imbedding  $W^{1,N}_0(B)\hookrightarrow L^{\infty,\infty;-1/N'}$ follows from the inequality \eqref{pointwise} for radial functions (see e.g. \cite{Moser}) for $p=N$, combined with the Polia-Szeg\"o inequality, and the H\"older inequality yields therefore the following family of imbedding into Lorentz-Zygmund spaces when $p=N$: 
$W^{1,N}_0(B)\hookrightarrow L^{\infty,q;-1/q-1/N'}$ for $p\le q\le \infty$. 

The smallest of the target spaces corresponds to $q=p$ in both cases, and, moreover, 
these imbeddings are optimal in the class of rearrangement-invariant spaces 
(see \cite{BS} for the definition), shown by Peetre \cite{Peetre} in the case
$N>p$ and by Brezis and Wainger \cite{BW} in the case $p=N$.

The limiting Sobolev inequality and the Trudinger-Moser
inequality are optimal in the sense that one cannot
replace the correspondent nonlinearity by any other with a faster growth.
 Indeed, if $h(s)$ is any continuous non-decreasing unbounded function on $[0,\infty)$, then 
\[
\sup_{u\in W^{1,N}_0(B), \|\nabla u\|_N\le 1}\int_B h(|u|)e^{\alpha_N |u|^{N'}}\dx=+\infty,
\]
and 
\[
 \sup_{u\in \mathcal D^{1,p}(\R^N),\|\nabla u\|_{p}\le 1}\int_\Omega h(|u|)
|u|^{p^*}\dx=+\infty.
\]

In the case $N>p$ this can be immediately seen by evaluating the functional
on $t^\frac{N-2}{2}u_0(tx)$ with $t>0$ and $u_0\neq 0$. 
In the case $p=N$, one arrives at the similar conclusion by evaluation of the functional on
the Moser family of functions \eqref{MF} 
normalized in the norm of $W_0^{1,N}(B)$. 
 
\section*{Appendix B. Moser functional in the radial case}
In this appendix we list some basic properties of Moser functional on the radial
subspace $W_{0,r}^{1,N}(B)$
of $W_0^{1,N}(B)$. In particular
we show that, in restriction to radial functions, it is weakly continuous on 
any sequence, that is not, asymptotically, a sequence of concentrating Moser functions.
This conclusion can be inferred from the original paper by Moser \cite{Moser}, 
while the notations we use here are brought from the paper \cite{AOT}. 
The calculations involved in this proof also allow to present the original Moser's proof
of the Trudinger-Moser inequality in a concise and streamlined form. 

Let $m_t$, $t\in(0,1)$, be the family of Moser functions \eqref{MF} and consider
the following
functional on $W_{0,r}^{1,N}(B)$:
\begin{equation}
\label{sp0}
\langle m_t^*,u\rangle\eqdef\int_B|\nabla m_t|^{N-2} \nabla m_t\cdot\nabla
u\dx,\;t\in(0,1). 
\end{equation}

An elementary computation shows that the functional $m_t^*$ is continuous and 
\begin{equation}
\label{sp}
\langle m_t^*,u\rangle = \omega_{N-1}^{1/N}\log(1/t)^{-1/N'}u(t),\;t\in(0,1).
\end{equation}

\begin{proposition}
\label{weakrad} Let $u_k\in W_{0,r}^{1,N}(B)$, $\|\nabla u_k\|_N\le 1$,
$u_k\rightharpoonup u$
and let $J$ be the Moser functional \eqref{MoserFunc}.
Then $J(u_k)\to J(u)$, unless the sequence $u_k$ has a renamed subsequence such
that 
$u_k-m_{t_k}\to 0$ in $W_{0,r}^{1,N}(B)$, with $t_k\to 0$.
\end{proposition}
\begin{proof}
Let us substitute \eqref{sp} into the definition of Moser functional. After
elementary simplifications one arrives 
the following
representation.
\begin{equation}
 \label{Moser-repres}
J(u)=\omega_{N-1}\left(\int_0^1 r^{N(1-\langle m_t^*,u\rangle^N)}\frac{\mathrm
dr}{r}-1/N\right),
\end{equation}
where $u\in W_{0,r}^{1,N}(B)$ and $\|\nabla u\|_N=1$.
Assume first that there exists $\epsilon >0$
such that $\langle m_t^*,u_k\rangle^N\le 1-\epsilon$. 
Then $J(u_k)\to J(u)$ by the Lebesgue dominated convergence theorem. 
The remaining case is when for some $t_k\in(0,1)$, $u_k-m_{t_k}\to 0$ in $W_{0,r}^{1,N}(B)$. 
Assume first that the weak limit $u$ is not zero. Then, necessarily, $u_k\to m_{t}$ in $W_{0,r}^{1,N}(B)$ 
for some $t\in(0,1)$. This implies the uniform convergence of $u_k$ on $[t,1]$ as well as 
$\int_{B_t}|\nabla u_k|^N\dx\to 0$, from which easily
follows $J(u_k)\to J(m_t)$. 
If $u_k=m_{t_k}+o(1)\rightharpoonup 0$ with $t_k\to 1$, an argument repetitive
of that for the case $u_k\to m_t$ above will give
$J(u_k)\to 0=J(u)$. We have, therefore, with necessity, a renamed
subsequence  $u_k=m_{t_k}+o(1)$ with $t_k\to 0$.
\end{proof}

Let

\begin{equation}
\label{gauge}
h_su(r)\eqdef s^{-1/N'}u(r^s), s>0, u\in W_{0,r}^{1,N}(B),
\end{equation}

Elementary calculations show that the operators \eqref{gauge} form a multiplicative group of linear isometries 
on $W_{0,r}^{1,N}(B)$. 
Furthermore, for every $s>0$ and $t\in(0,1)$,
\begin{equation}
\label{g_mu}
h_sm_t=m_{t^{1/s}}.
\end{equation}

We also note a well-known radial estimate: for each $u\in W_{0,r}^{1,N}(B)$,
\begin{equation}
\label{pointwise}
\sup_{r\in(0,1)}|u(r)|(\log(1/r)^{-1/N'}\le \omega_{N-1}^{-1/N} \|\nabla u\|_N.
\end{equation}




\end{document}